\documentclass[twoside, a4paper, 11pt]{amsart}

\title[ ]{Extensions of Schreiber's theorem on discrete approximate subgroups in $\mathbb{R}^d$}

\usepackage{amsaddr}
\author{Alexander Fish}
\address{School of Mathematics and Statistics F07, University of Sydney, NSW 2006, Australia}
\email{alexander.fish@sydney.edu.au}
\usepackage{amsfonts}
\usepackage{amsthm}
\usepackage{verbatim}
\usepackage{amsmath, amssymb}
\usepackage{tikz}
\usetikzlibrary{matrix, arrows}
\usepackage{listings}
\usepackage{color}
\usepackage{listings}
\usepackage[all]{xy}
\usepackage[pdftex,colorlinks,linkcolor=blue,citecolor=blue]{hyperref}
\usepackage{graphicx}
\usepackage{float}
\usepackage[margin=3cm]{geometry}
\usepackage{bigints}
\usepackage{dsfont}
\begin{document}
\maketitle
\raggedbottom

\newcommand{\cA}{\mathcal{A}}
\newcommand{\cB}{\mathcal{B}}
\newcommand{\cC}{\mathcal{C}}
\newcommand{\cD}{\mathcal{D}}
\newcommand{\cE}{\mathcal{E}}
\newcommand{\cF}{\mathcal{F}}
\newcommand{\cG}{\mathcal{G}}
\newcommand{\cH}{\mathcal{H}}
\newcommand{\cI}{\mathcal{I}}
\newcommand{\cJ}{\mathcal{J}}
\newcommand{\cK}{\mathcal{K}}
\newcommand{\cL}{\mathcal{L}}
\newcommand{\cM}{\mathcal{M}}
\newcommand{\cN}{\mathcal{N}}
\newcommand{\cO}{\mathcal{O}}
\newcommand{\cP}{\mathcal{P}}
\newcommand{\cQ}{\mathcal{Q}}
\newcommand{\cR}{\mathcal{R}}
\newcommand{\cS}{\mathcal{S}}
\newcommand{\cT}{\mathcal{T}}
\newcommand{\cU}{\mathcal{U}}
\newcommand{\cV}{\mathcal{V}}
\newcommand{\cW}{\mathcal{W}}
\newcommand{\cX}{\mathcal{X}}
\newcommand{\cY}{\mathcal{Y}}
\newcommand{\cZ}{\mathcal{Z}}
\newcommand{\bA}{\mathbb{A}}
\newcommand{\bB}{\mathbb{B}}
\newcommand{\bC}{\mathbb{C}}
\newcommand{\bD}{\mathbb{D}}
\newcommand{\bE}{\mathbb{E}}
\newcommand{\bF}{\mathbb{F}}
\newcommand{\bG}{\mathbb{G}}
\newcommand{\bH}{\mathbb{H}}
\newcommand{\bI}{\mathbb{I}}
\newcommand{\bJ}{\mathbb{J}}
\newcommand{\bK}{\mathbb{K}}
\newcommand{\bL}{\mathbb{L}}
\newcommand{\bM}{\mathbb{M}}
\newcommand{\bN}{\mathbb{N}}
\newcommand{\bO}{\mathbb{O}}
\newcommand{\bP}{\mathbb{P}}
\newcommand{\bQ}{\mathbb{Q}}
\newcommand{\bR}{\mathbb{R}}
\newcommand{\bS}{\mathbb{S}}
\newcommand{\bT}{\mathbb{T}}
\newcommand{\bU}{\mathbb{U}}
\newcommand{\bV}{\mathbb{V}}
\newcommand{\bW}{\mathbb{W}}
\newcommand{\bX}{\mathbb{X}}
\newcommand{\bY}{\mathbb{Y}}
\newcommand{\bZ}{\mathbb{Z}}
\newcommand{\eps}{\varepsilon}

\newcounter{dummy} \numberwithin{dummy}{section}

\theoremstyle{definition}
\newtheorem{mydef}[dummy]{Definition}
\newtheorem{prop}[dummy]{Proposition}
\newtheorem{corol}[dummy]{Corollary}
\newtheorem{thm}[dummy]{Theorem}
\newtheorem{lemma}[dummy]{Lemma}
\newtheorem{eg}[dummy]{Example}
\newtheorem{notation}[dummy]{Notation}
\newtheorem{remark}[dummy]{Remark}
\newtheorem{claim}[dummy]{Claim}
\newtheorem{Exercise}[dummy]{Exercise}
\newtheorem{question}[dummy]{Question}

\newtheorem*{thm*}{Theorem}

\begin{abstract} 
In this paper we give an alternative proof of Schreiber's theorem which says that an infinite discrete approximate subgroup in $\bR^d$ is relatively dense around a subspace. We also deduce from Schreiber's theorem two new results. The first one says that any infinite discrete approximate subgroup in $\bR^d$ is a restriction of a Meyer set to a thickening of a linear subspace in $\bR^d$, and the second one provides an extension of Schreiber's theorem to the case of the Heisenberg group.
\end{abstract}

\section{\textbf{Introduction}} 

 In this paper we study approximate subgroups. Recall that for a group $H$, a set $\Lambda \subset H$ is called an \textit{approximate subgroup} if
there exists a finite set $F \subset H$ such that $\Lambda^{-1}\Lambda \subset F \Lambda $, where $\Lambda^{-1} = \{ \lambda^{-1} \, | \, \lambda \in \Lambda \}$. In the case where $H$ is non-commutative, we will also assume as a part of the definition that $\Lambda$ contains the identity element of $H$ and $\Lambda$ is symmetric:
\begin{itemize}
\item $e_H \in \Lambda$,
\item $\Lambda^{-1} = \Lambda$.
\end{itemize}


Any finite  set in a group $H$ is contained in an approximate subgroup. An interesting question of classification of approximate subgroups arises if we  control the cardinality of $F$, while the cardinality of $\Lambda$ is finite but much larger than of the set of translates $F$, and in this case we say that $\Lambda$ has a small doubling. The classification of finite sets having small doubling for the ambient group $H = \bZ$ has been obtained by Freiman in his seminal work \cite{Fr}. These results have been eventually extended to all abelian groups by Green and Ruzsa \cite{GR}, and to arbitrary ambient groups by Hrushovski \cite{Hr}, and by Breuillard, Green and Tao \cite{BGT}.

We will investigate here infinite discrete approximate subgroups in  $\bR^d$ and in the Heisenberg group.
Infinite discrete relatively dense approximate subgroups in $\bR^d$, \textit{Meyer sets}, have been studied extensively by Meyer \cite{Me}, Lagarias \cite{La}, Moody \cite{Mo} and many others. 
It has been proved by Meyer \cite{Me} that a discrete relatively dense approximate subgroup in $\bR^d$ is a subset of a model (cut and project) set \cite{Mo}.
 Thus, despite a possible aperiodicity of Meyer sets, they all arise from lattices in (possibly) much higher dimensional spaces.  Very recently, there was a spark of interest in the extension of Meyer theory beyond the abelian case. A foundational work of Bj\"orklund and Hartnick \cite{BH} introduced the notion of an approximate lattice within Lie groups. The approximate lattices behave similarly to genuine lattices, and therefore, they are a good analog of Meyer sets in the non-abelian case.

The paper addresses a natural question of what kind of structure possesses an infinite discrete approximate subgroup $\Lambda$ in $\bR^d$ (or in the Heisenberg group) which is not relatively dense in the whole space. It has been almost forgotten by the mathematical community, that Schreiber in his thesis in 1972 \cite{S} proved that 
in the real case $\Lambda$ has to be relatively dense around a subspace, see definition \ref{rel-dense-def}. We provide an alternative, more geometric, proof of Schreiber's result. We also extend his theorem and show  that any discrete infinite approximate subgroup in $\bR^d$ is a subset of a Meyer set. In addition, we extend Schreiber's theorem to the case where the ambient space is the Heisenberg group. 
\medskip

\begin{remark}
The first draft of the paper dealt only with the real case (and had a different title). It was not known to the author that Theorem \ref{main-thm} was already proved by Schreiber. The author thanks Simon Machado for providing the reference to Schreiber's thesis. 
\end{remark} 

\begin{remark} After the first draft of the paper was released, Simon Machado has obtained the analog of Theorem \ref{heis-thm} to all connected real nilpotent groups. 
\end{remark}
\medskip

\textit{\textbf{Acknowledgment}.} The author is grateful to American Institute of Mathematics (AIM) and the organisers of the workshop on ``Nonstandard methods in combinatorial number theory" at AIM, where this project has been initiated. We also thank Terrence Tao who suggested the statement of Theorem \ref{cor} in the case $d = 2$. We would like also to thank Benji Weiss for fruitful discussions and anonymous referees that made very invaluable comments on the first draft of the paper, and, in particular, asked a question that led to Theorem \ref{complete-classification}. The paper has been influenced by Michael Bj\"orklund, and, particularly, by his series of lectures on quasi-crystals given at Sydney University in April 2016. We thank Michael for sharing his mathematical ideas with us. Finally, the author thanks Simon Machado for sharing with him his insights on the topic.
\medskip

\section{\textbf{Main Results}}
We will always assume that the underlying group $H$ possesses a left $H$-invariant metric $d_H$, and for any $r > 0$ and $h \in H$ we will denote by $B_r(h) = \{ g \in H \, | \, d_H(g,h) \le r\}$ the ball of radius $r$ around $h$. We will call a set $\Lambda \subset H$ \textit{discrete} if for every point $\ell \in \Lambda$ there exists $\delta = \delta(\ell)$ such that $B_{\delta}(\ell) \cap \Lambda = \{ \ell\}$. It is well known, that if $\Lambda \subset H$ is a discrete approximate subgroup, then $\Lambda$ is uniformly discrete, i.e., there exists $\delta > 0$ such that for all $\ell \in \Lambda$ we have $B_{\delta}(\ell) \cap \Lambda = \{\ell\}$. Indeed, $\Lambda$ is uniformly discrete if and only if $\Lambda^{-1}\Lambda$ does not contain identity $e_H$ as an accumulation point. 
Since $e_H \in \Lambda^{-1}\Lambda$, and $\Lambda^{-1}\Lambda$ is discrete, it follows that $e_H$ is not an accumulation point of $\Lambda^{-1}\Lambda$, and therefore $\Lambda$ is uniformly discrete. We will call $\Lambda$ \textit{relatively dense} (or $R$-relatively dense) if there exists $R > 0$ such that for every $h \in H$ we have $B_R(h) \cap \Lambda \neq \emptyset$.

 The following notion will play a key role in our paper.

\begin{mydef}
\label{rel-dense-def}
Let $H'$ be a subgroup of the group $H$. We will say that $\Lambda \subset H$ is \textit{relatively dense around $H'$} if there exists $R > 0$ such that:
\begin{itemize}
\item For every $h \in H'$ the ball of radius $R$ and centre $h$, i.e., $B_{R}(h) = \{ x \in H\, | \, d_H(x,h) \le R \}$, intersects non-trivially $\Lambda$.\\
\item The $R$-neighbourhood of $H'$ in $H$ contains $\Lambda$, i.e., 
\[
\Lambda \subset \bigcup_{h \in H'} B_{R}(h). 
\] 
\end{itemize} 
\end{mydef}

\subsection{Discrete approximate subgroups in $\bR^d$}

In this paper we give an alternative proof of Schreiber's theorem \cite{S} that discrete approximate subgroups in $\bR^d$ are relatively dense around some subspace. 

\begin{thm}\label{main-thm}[Schreiber, 1972]
Let $\Lambda \subset \bR^d$ be an infinite discrete approximate subgroup. Then there exists a linear subspace $L \subset \bR^d$ such that $\Lambda$ is relatively dense around $L$.
\end{thm}

As a corollary of Theorem \ref{main-thm} we obtain a complete characterisation of infinite approximate subgroups in $\bR^d$ in terms of Meyer sets. Recall, that a set in $\Lambda \subset \bR^d$ is a \textit{Meyer set} if 
\begin{itemize}
\item $\Lambda$ is discrete and relatively dense in $\bR^d$,
\item There exists a finite set $F \subset \bR^d$ such that $\Lambda - \Lambda \subset \Lambda + F$.
\end{itemize}

\begin{thm}\label{complete-classification}
A set $\Lambda \subset \bR^d$ is an infinite discrete approximate subgroup if and only if there exists a Meyer set $\Lambda' \subset \bR^d$, a subspace $L \subset \bR^d$ and $R > 0$ such that 
\[
\Lambda = \Lambda' \cap \left( L + B_R(0_{\bR^d})\right),
\]
and $\Lambda'$ is $R/2$-relatively dense in $\bR^d$.

\end{thm}
As another corollary of Theorem \ref{main-thm}, we obtain a complete characterisation of infinite approximate subgroups in $\bZ^d$.

\begin{thm}\label{cor}
Let $\Lambda$ be a subset in $\bZ^d$. The set $\Lambda$ is an infinite approximate subgroup if and only if there exists a linear subspace $L \subset \bR^d$ such that $\Lambda$ is relatively dense around $L$.
\end{thm}

%

A third application of Theorem \ref{main-thm} is that any discrete approximate subgroup in $\bR^d$  is ``very close'' to being a Meyer set on a subspace of $\bR^d$. More precisely, we prove the following result.

\begin{prop}\label{meyer}
Let $\Lambda \subset \bR^d$ be an infinite discrete approximate subgroup. Then there exist a subspace $L \subset \bR^d$ and $R > 0$ such that:
\begin{itemize}
\item The orthogonal projection $\Lambda_L$ of $\Lambda$ on the subspace $L$ is a Meyer set in $L$, i.e., $\Lambda_L$ is discrete relatively dense approximate subgroup in $L$, \\
\item $\Lambda \subset \Lambda_L + B_R(0_{\bR^d})$. 
\end{itemize}
\end{prop}
\medskip

The following example shows that an infinite discrete approximate subgroup $\Lambda$ is not necessarily subset of finitely many translates of $\Lambda_L$.

\begin{eg}
Let  $L = Span((1,\sqrt{3})) \subset \bR^2$, and let $\Lambda \subset \bR^2$ be the set of all $(m,n) \in \bZ^2$ such that $dist((m,n),L) \le 1$. Then $\Lambda$ is discrete since it is a subset of the integer lattice, and $\Lambda - \Lambda \subset \Lambda + F$ for a finite set\footnote{We can take $F = \bZ^2 \cap B_{2}(0_{\bR^2})$.} $F \subset \bZ^2$. But the orthogonal projection of $\Lambda$ on  the orthogonal complement of $L$ in $\bR^2$ is infinite, since the slope of $L$ is irrational. This implies that for any finite set $F \subset \bR^2$ we have
\[
\Lambda \not \subset \Lambda_L + F,
\]
where $\Lambda_L$ is the orthogonal projection of $\Lambda$ onto the line $L$.
\end{eg} 

\subsection{Discrete approximate subgroups in the Heisenberg group} For any $n \ge 1$ we define the Heisenberg group $H_{2n+1}$ by the following procedure. Assume that $\omega:\bR^{2n} \times \bR^{2n} \to \bR$ is a symplectic form, i.e., $\omega$ is a bilinear, anti-symmetric and non-degenerate form. Then the Heisenberg group $H_{2n+1} = \bR^{2n} \rtimes_{\omega} \bR$ is defined by $H_{2n+1} = \{ (v,z) \, | \, v \in \bR^{2n}, z \in \bR\}$, and the multiplication is given by 
\[
(v_1,z_1) \cdot (v_2,z_2) = \left(v_1+v_2, z_1 + z_2 + \frac{1}{2}\omega(v_1,v_2)\right).
\]
We will denote by $V$ the symplectic space $\bR^{2n}$, and by $Z$ the abelian subgroup $Z = \{ (0,z) \, | \, z \in \bR\}$. The subgroup $Z$ is the centre of $H_{2n+1}$. The Heisenberg group $H_{2n+1}$ is 2-step nilpotent. Indeed, for any two elements $h_1= (v_1,z_1), h_2 = (v_2,z_2) \in H_{2n+1}$, the commutator of $h_1$ and $h_2$ satisfies
\begin{equation}
\label{import_id}
[h_1,h_2] = (0, \omega(v_1,v_2)).
\end{equation}
It is easy to see that the Heisenberg group can be equipped with a left invariant metric. In the topology defined by this metric, the sequence $(v_n,z_n)$ in $H_{2n+1}$ converges to $(v,z)$ if and only if $v_n \to v$ in $V$ and $z_n \to z$ in $Z$.

We extend Schreiber's theorem to the Heisenberg case.

\begin{thm}
\label{heis-thm}
Let $\Lambda \subset H_{2n+1}$ be an infinite discrete approximate subgroup. Then there exists a connected non-trivial subgroup $H'$ in $H_{2n+1}$ such that $\Lambda$ is relatively dense around $H'$. Moreover, if $H'$ is non-abelian, then the projection of $\Lambda$ onto $V$ is discrete. 
\end{thm}

Let us denote by $\pi_V$ the projection from $H_{2n+1}$ onto $V$, i.e., $\pi_V((v,z)) = v$ for any $(v,z) \in H_{2n+1}$. The following example shows that we cannot improve the statement of the theorem.

\begin{eg}
Let $\Lambda = \{(m\sqrt{5} + n \sqrt{3},0), m) \, | \, m,n \in \bZ\}$. Then $\Lambda$ is a discrete subgroup in $H_3$. It is clear that $\pi_V(\Lambda)$ is dense within $L = \bR \times \{0\} \subset V$, and therefore it is non-discrete.
\end{eg}

%

It is easy to see that the projection on the $Z$-coordinate of a discrete approximate subgroup is not necessarily discrete. Indeed, we can find lattices in $H_3$ with a dense set of the $Z$-coordinates.

\begin{eg}
It is easy to see that 
\[
\Lambda = \left\{ \left((m,n), m\sqrt{5} + \frac{1}{2} \bZ\right) \, | \, m,n \in \bZ \right\}
\]
is a discrete co-compact subgroup in $H_3$ (equipped with the determinant on $\bR^2$ as the symplectic form). But the projection of $\Lambda$ on $Z$ is everywhere dense.
\end{eg} 

By the methods similar to the ones used to prove Theorem \ref{heis-thm}, we prove also the following claim.

\begin{prop}
\label{heis-prop}
Let $\Lambda \subset H_{2n+1}$ be a discrete approximate subgroup. If $\pi_V(\Lambda)$ is relatively dense in $V$, then $\Lambda$ is an approximate lattice, i.e., $\Lambda$ is relatively dense in $H_{2n+1}$.
\end{prop}

The analog of Theorem \ref{complete-classification} is not possible in the Heisenberg group:

\begin{prop}
\label{ful-class-impossible}
Let $\Lambda = \{(m\sqrt{5} + n \sqrt{3},0), m) \, | \, m,n \in \bZ\}$. Then there is no approximate lattice (discrete relatively dense approximate subgroup) in $H_3$  which contains $\Lambda$.
\end{prop}

\section{\textbf{Discrete approximate subgroups in $\bR^d$}}

Let $\Lambda \subset \bR^d$ be an approximate subgroup. By translating $\Lambda$ if necessary, we can assume that $0_{\bR^d} \in \Lambda$. 
Denote by $K = diam(F)$. Since, for any two $\ell_1,\ell_2 \in \Lambda$ we have $\ell_1 - \ell_2 \in \Lambda + F$, this implies that there exists $\ell \in \Lambda$ with $\ell_1 - \ell_2 \in B_K(\ell)$. If, we take $\ell_1 = 0$, we obtain the following property of $\Lambda$:

\begin{center}
(\textbf{A}) for every $\ell \in \Lambda$ there exists $\ell' \in B_{K}(-\ell) \cap \Lambda$.
\end{center}

By use of the property (A), if we take $ \ell_2 \in \Lambda$, then there exists $\ell_3 \in \Lambda$ such that $\ell_3 \in B_K(-\ell_2)$. Also for every $\ell_1 \in \Lambda$,there exists $\ell_4 \in \Lambda$ such that  $\ell_1 - \ell_3 \in B_K(\ell_4)$. Finally, this implies that $\ell_1 + \ell_2 \in B_{2K}(\ell_4)$. Thus, the following property holds for any approximate subgroup $\Lambda$ containing the neutral element:
\begin{center}
(\textbf{B}) for any $\ell_1, \ell_2 \in \Lambda$ there exists $\ell' \in B_{2K}(\ell_1 + \ell_2) \cap \Lambda$.
\end{center} 

We will call the property (A) the almost symmetry, and (B) the almost doubling.
We start with an easy observation which proves Theorem \ref{main-thm} in the case $d=1$.

\begin{prop}\label{simple-prop} 
Let $\Lambda \subset \bR$ be an infinite discrete approximate subgroup. Then $\Lambda$ is relatively dense.
\end{prop}
\begin{proof}
Assume that $\Lambda \subset \bR$ is an infinite approximate subgroup. Take $\ell \in \Lambda$ with $\ell > 3K$ (which exists by uniform discreteness of $\Lambda$). By the almost doubling property there exists $\ell_2 \in \Lambda$ with $\ell_2 \in [2\ell-2K, 2 \ell + 2K] \subset [\ell + K, 2 \ell+2K]$. Similarly, there exists $\ell_3 \in \Lambda \cap B_{2K}(\ell_2 + \ell)$. Therefore, $\ell_3 \in [\ell_2+K, \ell_2 + \ell + 2K]$. Assume that we already constructed $\ell_1 = \ell, \ell_2,\ldots,\ell_{n} \in \Lambda$ satisfying that $\ell_m + K \leq \ell_{m+1} \leq \ell_m + \ell + 2K$ for $m=1,\ldots,n-1 $. Then there exists $\ell_{n+1} \in \Lambda \cap [\ell_n + \ell - 2K, \ell_n + \ell + 2K]$. Therefore, we constructed an increasing sequence in $\Lambda \cap \bR_{+}$ with bounded gaps. By almost symmetry property of $\Lambda$, we also have in $\Lambda$ the elements $\{-\ell',-\ell_2',\ldots,-\ell_n',\ldots\}$ with $\ell' \in B_{K}(-\ell)$. This finishes the proof of the Proposition. 
\end{proof} 

A higher-dimensional case is much more subtle. 
An important role in the proof of Theorem \ref{main-thm} will play the set of asymptotic directions of the points in $\Lambda$. 
\begin{mydef}
Let $\Lambda \subset \bR^d$ be a uniformly discrete infinite set. We call 
\[
 D(\Lambda) = \{ u \in S^{d-1} \, | \, \mbox{ there exists } (\ell_n) \in \Lambda \mbox{ with } \frac{\ell_n}{\| \ell_n \|} \to u \mbox{ and } \Vert \ell_n \Vert \to \infty\}
\] 
the \textit{set of asymptotic directions} of $\Lambda$.
\end{mydef}

It is easy to see that $D(\Lambda)$ is non-empty closed set. It will be very convenient to us to introduce the subspace generated by $D(\Lambda)$. Let $L \subset \bR^d$ be the smallest linear subspace with the property that $D(\Lambda) \subset L$. In other words, we have
\[
L = Span(D(\Lambda)).
\]
  
The next lemma is an important ingredient in the proof of  Theorem \ref{main-thm}. 

\begin{lemma}\label{lemma}
Assume that $\Lambda$ is an infinite discrete approximate subgroup. Let $L = Span(D(\Lambda))$ be a proper subspace in $\bR^d$. Then there exists $R > 0$ ($R = 3\cdot diam(F)$) such that  
\[
\Lambda \subset \bigcup_{x \in L} B_R(x).
\]
\end{lemma}
\begin{proof}
Let $\Lambda \in \bR^d$ be an infinite discrete approximate subgroup, i.e., there exists a finite set $F \subset \bR^d$ with $\Lambda - \Lambda \subset \Lambda + F$. Denote by $K = diam(F)$.
For any $\eps > 0$ and any $u \in S^{d-1}$ we define the cone 
\[
V_{\eps}(u) = \{ t v \, | \, t > 0, v \in S^{d-1} \mbox{ with } \langle v, u \rangle \ge 1 - \eps \}.
\]
Let us take $R = 3K$. We claim that 
\[
\Lambda \subset \bigcup_{x \in L} B_R(x).
\]
Indeed, if there exists $\ell \in \Lambda$ such that $\ell \not \in \bigcup_{x \in L} B_R(x)$, let us define $u = \frac{\ell}{\| \ell \|}$ and $1 - \eps = \frac{\sqrt{\| \ell \|^2 - 5K^2}}{\| \ell \|}$. Then we construct a sequence $\ell_1, \ell_2,\ell_3,\ldots$ in $\Lambda$ with $\ell_n \to \infty$ and $\ell_n \in V_{\eps}(u)$. Since, clearly, we have
\[
V_{\eps}(u) \cap L = \{ 0_{\bR^d} \},
\] 
this will imply the contradiction. 
\medskip

The construction is the same as in the proof of Proposition \ref{simple-prop}. Let us define $\ell_1 = \ell$. We find $\ell_2 \in B_{2K}(\ell_1 + \ell) \cap \Lambda$. The following calculation guarantees that $\ell_2 \in V_{\eps}(u)$:
\[
\left\langle \frac{\ell_2}{\| \ell_2 \|}, u \right\rangle \ge \frac{2 \| \ell \|}{\sqrt{4\| \ell \|^2+ 4 K^2}} = \frac{\| \ell \|}{ \sqrt{\|\ell \|^2 + K^2}} \ge 1 - \eps.
\]
Also, it is clear that $\|\ell_2 \| \ge \| \ell_1 \| +K$. Assume that we constructed a finite sequence $\ell_1,\ell_2, \ldots, \ell_n \in \Lambda$ with $\| \ell_{m+1} \| \ge \| \ell_m \| + K$, $m=1,\ldots,n-1$, and $\ell_1,\ell_2,\ldots,\ell_n \in V_{\eps}(u)$. Then there exists $\ell_{n+1} \in B_{2K}(\ell_n + \ell) \cap \Lambda$. Clearly, we have 
\[
\| \ell_{n+1} \| \ge \| \ell_n \| + K.
\] 
Finally, for any vector $v \in V_{\eps}(u)$ we have
\[
B_{2K}(v + \ell) \subset V_{\eps}(u).
\]
This will guarantee that $\ell_{n+1} \in V_{\eps}(u)$. Indeed, if a vector $v \in V_{\eps}(u)$, 
then $v+ V_{\eps}(u) \subset V_{\eps}(u)$, and therefore we have:
\[
dist(v+\ell, \partial V_{\eps}(u)) \ge dist(v + \ell, \partial(v + V_{\eps}(u)))
\]
\[
 = dist(\ell, \partial(V_{\eps}(u))) = \| \ell \| (1-\eps) 
= \sqrt{\| \ell \|^2 - 5K^2} > 2K. 
\]
\end{proof}

Our next step in the proof of Theorem \ref{main-thm} is to construct a system of ``basis" vectors for $\Lambda$. Let $L = Span(D(\Lambda))$, and let $R$ satisfy
\begin{equation}\label{R-thick}
\Lambda \subset \bigcup_{x \in L} B_{R}(x).
\end{equation}
Assume that $dim(L) = k$, where $1 \le k \le d$, and denote by $K = diam(F)$. By the definition of the set of asymptotic directions $D(\Lambda)$, there exists $\eps > 0$ such that for every $M > 0$ there exist $k$ elements $\ell_1,\ldots,\ell_k \in \Lambda$ satisfying the following properties:
\medskip

\begin{itemize}
\item{(\textbf{$\eps$-well spreadness})}
For all $1 \le i \le k$, any $v_i \in B_{2K}(\ell_i)$, and $v_j \in B_{2K}(\varepsilon_j\ell_j), j \neq i, \varepsilon_j \in \{-1,1\}$, let us denote by $\gamma_i$ the angle between $v_i$ and the subspace 
\[
V^i = Span\{v_1,\ldots,v_{i-1},v_{i+1},\ldots,v_k\}.
\]
 Then we require:
\[
\eps \le \gamma_i \le \pi - \eps,
\]\\
\item{(\textbf{no short vectors})} For every $1 \le i \le k$ we have 
\[
\| \ell_i \| \ge M.
\]
\end{itemize}
By almost symmetry of $\Lambda$, we can also find the ``reflected" vectors $\{\ell_1',\ldots,\ell_k'\} \subset \Lambda$ which satisfy the property
\[
\ell_i' \in B_K(-\ell_i),\mbox{ } i=1,\ldots,k.
\]
 Let us denote by $\cF = \{\ell_1,\ldots,\ell_k, \ell_1',\ldots,\ell_k'\}$. By Lemma \ref{lemma} there exists $R > 0$ such that $\Lambda \subset L_R$, where    
$L_R = \bigcup_{x \in L} B_R(x)$ is the $R$-thickening of the subspace $L$. Let us assume that $R \ge K$.
Finally, for any choice of $M > 0$, let us call the corresponding system $\cF$ the $(M,\eps,L,R)$-system in $\bR^d$, and denote by $T(\cF) = \max\{ \| \ell_i \| \, | \, i=1,\ldots,k \}$.

Our next claim is the following.

\begin{prop}\label{prop} 
Let $\eps > 0$.
 There exist $\delta = \delta(\eps) > 0$ and $M_0$ such that if 
$\cF$ is a $(M,\eps,L,R)$-system for  a subspace   $L$ of $\bR^d$, $M \ge M_0$ and  $R \ge K$, then for all $x \in L_R$ with $\| x \|$ large enough  there exists $\ell \in \cF$ such that for every $v \in B_K(\ell)$
we have
\[
\| x - v\| \le \| x \| - \frac{\delta M}{4}.
\]
\end{prop}
\begin{proof}
Let us first assume on the system $\cF$ the following:
\begin{itemize}
\item $\cF$ is symmetric, i.e., if $\ell \in \cF$ then $-\ell \in \cF$\\ 
\item $\cF \subset L.$
\end{itemize}
We will also assume that  $x \in L$ and $v \in \cF$.

Our next step is to observe that there exists $\delta = \delta(\eps) > 0$ such that for any $z \in S(L) = \{ x \in L \, | \, \| x \| = 1 \}$ there exists $v \in \cF$ with $| \langle z, v \rangle | \ge \delta \| v \|$. Indeed, we can assume that all the $\ell_i \in \cF$ are of length one. Denote by $S'$ the set of $k-$tuples $\{\ell_1,\ldots,\ell_k\}$ in $S(\bR^d)$ which are $\eps$-well spread. Since it is a closed condition, the set $S'$ is closed. By compactness of 
\[
U = \{ (z,\ell_1,\ldots,\ell_k) \, | \, z \in Span(\ell_1,\ldots,\ell_k), \| z \| = 1, (\ell_1,\ldots,\ell_k) \in S' \}
\]
 it follows that there exist $(\ell_1',\ldots,\ell_k') \in S'$,  $z_0 \in Span(\ell_1',\ldots,\ell_k')$ with $\| z_0 \| = 1$,  and $1 \le i_0 \le k$ such that 
\[
\min_{\{z,\ell_1,\ldots,\ell_k\} \in U} \max_{1 \le i \le k} | \langle z, \ell_i \rangle | =  \max_{1 \le i \le k} | \langle z_0, \ell_i' \rangle | = | \langle z_0,\ell_{i_0}' \rangle| . 
\]
Obviously, the right hand side is positive, since otherwise, we will have that $z_0 \not \in Span(\ell_1',\ldots,\ell_k')$. Then we define $\delta= | \langle z_0,\ell_{i_0}' \rangle |$.

Let $x \in L$ and let us consider the triangle with the vertices at the origin, $x$ and at $v \in \cF$ with\footnote{Since $\cF$ is symmetric, such $v$ exists.}  $\langle x , v \rangle \ge \delta \| x \| \| v \|$.
Denote by $D = \| v \|$. Notice that $D \le T(\cF)$. We have
\[
\|x-v\|^2 = \|x\|^2 + D^2 - 2 \langle x, v \rangle.
\]
Assume that $\| x \|$ satisfies:
\[
2 \delta \| x \| - (T(\cF))^2 \ge \delta \|x \|,
\]
and
\[
\| x \| \ge T(\cF).
\]
Then we have
\[
 \|x\| - \|x-v\| = \frac{2 \langle x, v \rangle - D^2}{\| x \| + \| x- v \|} \ge \frac{\delta \| x \| D}{3 \| x \|} \ge \frac{\delta M}{3}.
 \]
 
For a general $(M,\eps,L,R)$-system $\cF$ we can find a symmetric $(M,\eps/2,L,R)$-system $\cF'$ with $\cF' \subset L$, such that for every $\ell' \in \cF'$ there exists $\ell \in \cF$ with $\| \ell - \ell' \| \le R+K$. Take $x \in L_R$ with $\| x \|$ large. Then there exists $x' \in L$ such that $\|x - x'\| \le R$. By the previous discussion, there exists $\delta = \delta(\eps/2)$ such that for any $x' \in L$ there exists $\ell' \in \cF'$ with
\[
\| x' \| - \| x' - \ell'\| \ge \frac{\delta M}{3}.
\]
Take $\ell \in \cF$ such that $\|\ell - \ell' \| \le R+K$. Then for every $v \in B_K(\ell)$ we have
\[
\|x\| - \| x - v\| \ge (\|x' \| - \| x' - x \|) - \left( \| x - x' \| + \| x' - \ell'\| + \| \ell' - \ell \| + \| \ell - v \| \right)  
\]
\[
\ge \left( \| x' \| - \| x' - \ell' \| \right) - (2 \| x - x' \| + \| \ell' - \ell \| + \| \ell - v \|) \ge \frac{\delta M}{3} - 3R - 2K \ge \frac{\delta M}{4},
\]
where the last transition is correct if $M$ is large enough\footnote{Here we use that $\delta$ is independent of $K,R$ and $M$.}.
\end{proof}
\medskip

\subsection{\textbf{Proof of Theorem \ref{main-thm}}}
 Assume that $\Lambda \subset \bR^d$ is an infinite discrete approximate subgroup satisfying $\Lambda - \Lambda \subset \Lambda + F$ for a finite set $F$. Denote by $K =diam(F)$ and by $L = Span(D(\Lambda))$. Then by Lemma \ref{lemma} there exists $R > 0$ such that $\Lambda \subset L_R = \bigcup_{x \in L} B_R(x)$. By the discussion above, there exists $\eps> 0$ such that for an arbitrary $M > 0$ there exists $(M,\eps,L,R)$-system $\cF$ within $\Lambda$.  
Let us take $M > 0$ so large that the claim of Proposition \ref{prop} holds true for some $\delta = \delta(\eps) > 0$. Let $R'$ be such that 
for every $x \in L_R$ with $\| x \| \ge R'$ there exists $\ell \in \cF$ with the property that for every $v \in B_{K}(\ell)$  we have: 
\[
\|x - v\| \le \| x \| - \frac{\delta M}{4}. 
\]
We will show that for every $z \in L_R$ we will have $B_{R'}(z) \cap \Lambda \neq \emptyset$. Assume, on the contrary, that there exists $z \in L_R$ such that  $B_{R'}(z) \cap \Lambda = \emptyset$. Take minimal $R_2 > R'$ such that $B_{R_2}(z) \cap \Lambda \neq \emptyset$. This means that for every $r < R_2$ we have $B_{r}(z) \cap \Lambda = \emptyset$, and that there exists $y \in B_{R_2}(z) \cap \Lambda$. 

Let us denote by $x = z-y$. Then $\| x \| = R_2$, and therefore there exists $\ell \in \cF \subset \Lambda$ such that for every $v \in B_K(\ell)$ we have
\[
\| x - v \| \le \| x \| - \frac{\delta M}{4} < \| x \|=R_2.
\] 
But, since $\Lambda$ is an approximate subgroup with $diam(F) = K$, we have that there exists $v \in B_K(\ell)$ such that $y + v  \in \Lambda$. This implies:
\[
\| z- (y+v) \| < R_2.
\]
Therefore, there exists $r < R_2$ such that $B_{r}(z) \cap \Lambda \neq \emptyset$. So, we get a contradiction. Therefore, indeed, for every $x \in L_R$ we have $B_{R'}(x) \cap \Lambda \neq \emptyset$. This finishes the proof of the theorem. 

\qed
\medskip

\subsection{\textbf{Proof of Theorem \ref{complete-classification}}}

``$\Rightarrow$": If $\Lambda \subset \bR^d$ is an infinite discrete approximate subgroup, then by Theorem \ref{main-thm} the set $\Lambda$ is relatively dense around a certain subspace $L \subset \bR^d$. Therefore, there exists $R > 0$ such that 
\begin{itemize}
\item $\Lambda \subset L + B_R(0_{\bR^d})$,
\item For every $z \in L$ we have $\Lambda \cap B_R(z) \neq \emptyset$.
\end{itemize}
Let $L^{\perp} \subset \bR^d$ be the orthogonal complement of $L$. Let $\Gamma \subset L^{\perp}$ be a lattice such that for any $x,y \in \Gamma$ we have $\| x - y \|_{L^{\perp}} \ge 4R$. Denote by $\Lambda ' = \Lambda + \Gamma$. Obviously, $\Lambda = \Lambda' \cap (L+B_R(0_{\bR^d}))$. We claim that  $\Lambda'$ is a Meyer set in $\bR^d$, i.e., discrete relatively dense approximate subgroup. Indeed, first notice that 
\[
\Lambda' - \Lambda' = (\Lambda - \Lambda) + (\Gamma - \Gamma) \subset (\Lambda + \Gamma) + F = \Lambda' + F,
\]
for some finite set $F \in \bR^d$. Also, $\Lambda'$ is discrete, since all different translates of $\Lambda$ by elements of $\Gamma$ are far apart. If $\lambda_1,\lambda_2  \in \Lambda'$, assume that $\lambda_1 = \ell_1+\gamma_1, \lambda_2 = \ell_2 + \gamma_2$ for $\ell_i \in \Lambda$, $\gamma_i \in \Gamma$, $i=1,2$, then 
\[
\lambda_1 - \lambda_2 = (\ell_1-\ell_2) + (\gamma_1-\gamma_2).
\]
If $\gamma_1 \neq \gamma_2$, then $\| \lambda_1 - \lambda_2 \| \ge 2R$. And in the case  $\gamma_1 = \gamma_2$ we use the uniform discreteness of $\Lambda$ to obtain a uniform bound on $\| \lambda_1 - \lambda_2 \|$, for $\lambda_1 \neq \lambda_2$. Finally, the relative density of $\Lambda'$ follows immediately from the relative density of $\Lambda$ around the subspace $L$ and the relative density of $\Gamma$ inside $L^{\perp}$.
\medskip

``$\Leftarrow$": Let $\Lambda' \subset \bR^d$ be a Meyer set. Let $R > 0$  be such that for any $x \in \bR^d$ we have $B_{R/2}(x) \cap \Lambda' \neq \emptyset$. Take any linear subspace $L \subset \bR^d$. Denote by $\Lambda = \Lambda' \cap (L + B_R(0_{\bR^d}))$. Then $\Lambda$ is an infinite discrete approximate subgroup. The only non-trivial claim is that $\Lambda$ is an approximate subgroup. To prove it, we will use Lagarias' theorem saying that if $\Lambda''$ is relatively dense in  $\bR^d$ and  $\Lambda'' - \Lambda''$ is uniformly discrete, then $\Lambda''$ is an approximate subgroup, i.e., there exists a finite set $F \subset \bR^d$ such that $\Lambda'' - \Lambda'' \subset \Lambda'' + F$. First, we construct such $\Lambda''$. Take a lattice $\Gamma \subset L^{\perp}$ satisfying that for any distinct $\gamma_1,\gamma_2 \in \Gamma$ we have $\| \gamma_1 - \gamma_2 \| \ge 4R$. Then define $\Lambda'' = \Lambda + \Gamma$. Obviously, $\Lambda''$ is relatively dense in $\bR^d$. We also have:
\[
\Lambda'' - \Lambda'' \subset (\Lambda' - \Lambda') \cap (L + B_{2R}(0_{\bR^d})) + \Gamma.
\]
This implies that $\Lambda'' - \Lambda''$ is uniformly discrete, and therefore, by Lagarias theorem,  there exists  a finite set $F \subset \bR^d$ with $\Lambda'' - \Lambda'' \subset \Lambda'' + F$.  
The latter implies that 
\[
\Lambda - \Lambda + \Gamma \subset \Lambda + \Gamma + F. 
\]
We claim that there exists $F' \subset \bR^d$ finite such that $\Lambda - \Lambda \subset \Lambda + F$. Indeed, for any  $\ell_1,\ell_2 \in \Lambda$ there exist $\ell_3 \in \Lambda, \gamma \in \Gamma$ and $f \in F$ such that 
\[
\ell_1 - \ell_2 = \ell_3 + \gamma + f.
\]
But the projection of  $\Lambda - \Lambda - \Lambda - F$ onto $L^{\perp}$ is at bounded distance from the origin, i.e., there exists $R' > 0$ such that  $\pi_{L^{\perp}}(\Lambda - \Lambda - \Lambda - F) \subset B_{R'}(0_{\bR^d}) \cap L^{\perp}$, where the operator $\pi_{L^{\perp}}$ is the orthogonal projection onto $L^{\perp}$. This implies that every such $\gamma \in \Gamma$ for which there exist $\ell_1,\ell_2,\ell_3 \in \Lambda$ and $f \in F$ with $\gamma = \ell_1 - \ell_2 - \ell_3 - f$ is at bounded distance from the origin in $L^{\perp}$. But there are only finitely many $\gamma \in \Gamma$ which lie in the ball 
$B_{R'}(0_{\bR^d}) \cap L^{\perp}$. Denote by $F_2$ the finite set $F_2 = \Gamma \cap B_{R'}(0_{\bR^d})$, and by $F' = F + F_2$. Then we have 
\[
\Lambda - \Lambda \subset \Lambda + F'.
\]
\qed
\medskip

\subsection{\textbf{Proof of Theorem \ref{cor}}}

It follows immediately from Theorem \ref{main-thm} that if $\Lambda \subset \bZ^d$ is an infinite approximate group, then there exists a subspace $L \subset \bR^d$ and $R > 0$ such that $\Lambda \subset L + B_R(0_{\bR^d})$, and for every $\ell \in L$ we have that $\Lambda \cap B_R(\ell) \neq \emptyset$. Let us call any $\Lambda$ that satisfies these constraints with respect to a subspace $L$ as being \textit{relatively dense around $L$}.  
\medskip

On the other hand, assume that $\Lambda \subset \bZ^d$ is relatively dense around a subspace $L \subset \bR^d$. We will show that such $\Lambda$ is necessarily an approximate subgroup. 

Indeed, let us first take $R_1 > 0$ with the property\footnote{We can take any $R_1 >\frac{\sqrt{d}}{2}$.} that 
for any point $x \in \bR^d$ we have $B_{R_1}(x) \cap \bZ^d \neq \emptyset$. Since, for any $\lambda \in \Lambda$ there exists $\ell \in L$ such that $\lambda \in B_R(\ell)$, we have that for any $\lambda_1,\lambda_2 \in \Lambda$ there exist
$x_1,x_2 \in \bZ^d \cap L+B_{R_1}(0)$ such that 
\[
\lambda_i \in B_{R+R_1}(x_i), \mbox{ for } i = 1,2.
\]
Therefore, there exist $f_1,f_2 \in B_{R+R_1}(0) \cap \bZ^d$ such that
\[
\lambda_i = x_i + f_i, \mbox{ for } 1 = 1,2.
\]
Also, notice that $x_1 - x_2 \in L + B_{2R_1}(0)$. Therefore, there exists $\lambda \in \Lambda$ such that $x_1 - x_2 \in B_{3R}(\lambda)$. Thus, there exists $f' \in B_{3R}(0) \cap \bZ^{d}$ such that $x_1-x_2 = \lambda + f'$. Finally, 
let us denote by $F = B_{5R + 2R_1}(0) \cap \bZ^d$ (finite set). Then we have
\[
\lambda_1 - \lambda_2 = (x_1 + f_1) - (x_2 + f_2) = (x_1 - x_2) + (f_1 - f_2) = \lambda + (f_1 - f_2 + f') \in \Lambda + F.
\]
This finishes the proof of the Theorem.

\qed
\medskip

\subsection{\textbf{Proof of Proposition \ref{meyer}}}

Let $\Lambda$ be a discrete approximate subgroup in $\bR^d$. By Theorem \ref{main-thm} we know that there exist a subspace $L$ and $R > 0$ such that $\Lambda$ is relatively dense around $L$, i.e., $\Lambda \subset L + B_R(0_{\bR^d})$ and for any $x \in L$ we have $B_R(x) \cap \Lambda \neq \emptyset$. Let us denote by $\pi$ the orthogonal projection from $\bR^d$ to $L$. And let $\Lambda_L = \pi(\Lambda)$. 
\medskip

By linearity of the map $\pi$ we get that $\Lambda_L$ is an approximate subgroup. For $\ell_1,\ell_2 \in \Lambda_L$  there exist $\lambda_1,\lambda_2 \in \Lambda$ such that $\ell_i = \pi(\lambda_i), i = 1,2$. Denote by $L^{\perp}$ the orthogonal complement to $L$, i.e., we have $\bR^d = L \oplus L^{\perp}$. Then there exist $\mu_1,\mu_2 \in L^{\perp}$ such that 
\[
\lambda_i = \ell_i + \mu_i, \mbox{ for } i = 1,2.
\]
But $\Lambda$ is an approximate subgroup. Therefore, there exists a finite set $F \subset \bR^d$ such that $\Lambda - \Lambda \subset \Lambda + F$. This implies that there exist $\lambda \in \Lambda,$ and $f \in F$ such that
\[
\lambda_1 - \lambda_2 =\lambda + f.
\]
By projecting both sides on $L$ we obtain:
\[
 \ell_1 - \ell_2 = \pi(\lambda) + \pi(f).
\]
Let us denote $F' = \pi(F)$ (a finite set). Then we have
\[
\Lambda_L - \Lambda_L \subset \Lambda_L + F'.
\] 

We also have that $\Lambda_L$ is relatively dense in $L$ since $L \subset \Lambda_L + B_{2R}(0_{\bR^d})$.
\medskip

The set $\Lambda_L$ is discrete. Indeed, assume that it is not discrete. Then there exists $(\ell_n) \subset \Lambda_L$ with $\ell_n \to x \in L$ and $\ell_n \neq x$ for every $n$. Let $(\mu_n) \subset L^{\perp}$ such that $\lambda_n = \ell_n + \mu_n \in \Lambda$. Since all $\mu_n$ are bounded, then there is a convergent subsequence $(\mu_{n_k})$. Denote its limit by $\mu \in L^{\perp}$. Then we have 
\[
\lambda_{n_k} = \ell_{n_k} + \mu_{n_k} \to x + \mu.
\]
Since $\Lambda$ is discrete, this implies that the sequence $\lambda_{n_k}$ is fixed for $k$ large enough. This implies that the subsequence $\ell_{n_k}$ is fixed for $k$ large enough and we get a contradiction.
\medskip

All this together, shows that the set $\Lambda_L \subset L$ is a Meyer set. Finally, by the construction we have $\Lambda \subset \Lambda_L + B_R(0_{\bR^d})$.

\qed

\section{\textbf{Discrete approximate subgroups in the Heisenberg group}}

Assume that $n \ge 1$, and $\Lambda \subset H_{2n+1}$ is a discrete infinite approximate subgroup. Denote by $\Lambda_V = \pi_V(\Lambda)$. Our fist claim follows from the definition of an approximate group and the linearity of the projection operator $\pi_V$

\begin{lemma}
\label{heis-lem-1}
The set $\Lambda_V$ is an approximate subgroup in $V$.
\end{lemma}

Since the proof of Theorem \ref{main-thm} does not use the discreteness of an approximate subgroup in $V$ but only its unboundness, we derive that there exists a linear subspace $L \subset V$ such that $\Lambda_V$ is relatively dense around $L$. Our next claim will use the identity (\ref{import_id}).

\begin{lemma}
\label{lem2}
If $\omega(\Lambda_V,\Lambda_V) \neq 0$ and $\dim L \ge 1$, then $[\Lambda,\Lambda]$ is relatively dense in $Z$, and $\Lambda_V$ is discrete in $V$.
\end{lemma}
\begin{proof}
By the identity (\ref{import_id}), it follows that for any two elements $\lambda_1 = (v,z), \lambda_2 = (u,t)$, their commutator 
\begin{equation}
\label{ident}
[\lambda_1,\lambda_2] = (0, \omega(v,u)).
\end{equation}
 Also, by the assumptions of the lemma, there exist a line $L_0$ in $V$, and $R > 0$ such that for every $\ell \in L_0$ there exists $v_{\ell} \in \Lambda_V$ with $\| \ell - v_{\ell} \|_{V} \le R$. It is also clear from the assumptions that there exists $v \in \Lambda_V$ such that $v \not \in L_0$. Then it follows from the continuity of the symplectic form $\omega$ that the set 
\[
\left\{ \omega(v,u) \, | \, u \in \Lambda_V \right\}
\]
is relatively dense in $\bR$. The identity (\ref{ident}) implies that $[\Lambda,\Lambda]$ is relatively dense in $Z$. 
The only remaining part of the lemma that we have to prove is the discreteness of $\Lambda_V$.
Since $\Lambda$ is an approximate group in $H_{2n+1}$, it follows that there exists a finite set $F' \subset H$ ($F' = F F F$) such that 
\[
[\Lambda,\Lambda] \subset F' \Lambda.
\]
Thus there exists a relatively dense sequence $(t_n) \subset \bR$ such that for every $n$ corresponds at least one 
$f_n$ from the finite set  $F'^{-1}$ with $f_n(0,t_n) \in \Lambda$.  
Assume that $\Lambda_V$ is non-discrete. Then there exists a sequence $(v_n,z_n) \in \Lambda$ with $v_n \to v$  such that $v_n \neq v$ for all $n$. Then by applying from the left the elements $f_n (0,t_n)$ with $t_n+z_n$ is in a compact set in $\bR$ we have the new sequence 
\[
f_n (v_n, t_n+z_n) \subset F'^{-1}F \Lambda.
\]
But now we achieved that the new sequence is inside a compact set in $H_{2n+1}$. Thus, without loss of generality, we assume that the sequence $f_n (v_n, t_n + z_n)$ converges. 
Since $f_n$'s belong to a finite set $F'^{-1}$, by taking a subsequence, we can assume that  $f_n = f$ and $(v_n,t_n+z_n)$ converges to $(v,t)$ for some $t \in \bR$. Since the element $f_n$ is fixed, there exists a finite set $F''$ such that $(v_n,t_n+z_n) \subset F'' \Lambda$. But the set on the right hand side is discrete, while the sequence on the left hand side is not. We get a contradiction and it finishes the proof of the lemma.
\end{proof}

\subsection{\textbf{Proof of Theorem \ref{heis-thm}}}
Let $\Lambda$ be an infinite discrete approximate subgroup in the Heisenberg group $H_{2n+1}$. As we already noticed, the projection $\Lambda_V$ of $\Lambda$ onto $V$ is relatively dense around a subspace $L \subset V$. 
If $L = \{ 0 \}$, then by the boundness of $\Lambda_V$ and using the same reasoning as in the proof of Proposition \ref{simple-prop} we obtain that $\Lambda$ is relatively dense around the centre $Z$ of $H_{2n+1}$. Now assume that $\dim L \ge 1$. Then there are two cases:
\medskip

\begin{itemize}
\item[(1)] $\omega(\Lambda_V, \Lambda_V) = 0$,\\
\item[(2)] $\omega(\Lambda_V,\Lambda_V) \neq 0$.
\end{itemize} 
\medskip

In the first case, there exists a Lagrangian subspace $L' \subset V$ such that $\Lambda_V \subset L'$, and $\omega(L',L') = 0$. Then we make use of Schreiber's theorem with respect to the abelian group $V' = L' \times Z$ and conclude that there exists a subspace $L'' \subset V'$ such that $\Lambda$ is relatively dense around $L''$. This abelian subgroup $L''$ is clearly a connected subgroup of $H_{2n+1}$.
\medskip

In the second case, we invoke Lemma \ref{lem2} and obtain that $\Lambda$ is relatively dense around the connected subgroup 
$H' = LZ$, where 
\[
LZ = \{ (v,z) \, | \, v \in L, z \in \bR \}.
\]

To prove the last part of the theorem, we notice that $H'$ around which the subgroup $\Lambda$ is relatively dense is non-abelian only in the last case, i.e., $H' = \{ (v,z) \, | \, v \in L, z \in \bR \}$, and $\omega(L,L) \neq 0$. Then by Lemma \ref{lem2} we are done. 
\qed

%

\subsection{\textbf{Proof of Proposition \ref{heis-prop}}} If $\Lambda_V = \pi_V(\Lambda)$ is relatively dense in $V$, then $\omega(\Lambda_V, \Lambda_V) \neq 0$. By Lemma \ref{lem2}, we get that $[\Lambda,\Lambda]$ is relatively dense in $Z$. This easily implies the conclusion of the proposition.
\qed

\subsection{\textbf{Proof of Proposition \ref{ful-class-impossible}}}
Let $\Lambda$ be as in the statement of the proposition. Assume that there exists $\Lambda' \subset H_3$ such that $\Lambda \subset \Lambda'$ and $\Lambda'$ is relatively dense in $H_3$. Then the projection $\Lambda'_V$ of $\Lambda'$ onto $V$ is non discrete. On other hand, it follows from Lemma \ref{lem2} that $\Lambda'_V$ is discrete. We get a contradiction.
\qed

\end{document}